\documentclass[indent=latex,review=false,theme=default,marginsize=3cm,fontsize=12pt]{acmhomework}

\title{Pushforwards in Inverse Homotopical Diagrams}

\newcommand{\hM}{\widehat{M}}

\begin{abstract}
  We establish a sufficient condition for the category of homotopical inverse diagrams to be closed under pushforward inside the category of inverse diagrams in a fibration category.
\end{abstract}

\begin{document}

\author{Krzysztof Kapulkin \and Yufeng Li}

\maketitle

\section*{Introduction}
Developed originally to study generalized sheaf cohomology \cite{brown:aht}, Brown's theory of categories of fibrant objects has seen renewed interest in recent years coming from such disparate areas as: higher category theory \cite{szumilo:agt,szumilo:hha}, dependent type theory, and graph theory.
The structure of a category of fibrant objects seems to be the exact structure possessed by various examples appearing naturally in these contexts, for example, the classifying category of a dependent type theory \cite{avigad-kapulkin-lumsdaine,shu15} and the category of simple graphs with A-weak equivalences \cite{carranza-kapulkin}.

When applying this theory in concrete cases, one often works with categories of diagrams.
Namely, given a category of fibrant objects $\bC$ and a small category $\McI$, one can ask whether the category $\bC^\McI$ of $\McI$-diagrams in $\bC$ is again a category of fibrant objects.
This, of course, requires putting some restrictions on $\McI$ and, possibly, on the kind of diagrams one considers.
The most common of these is the requirement that $\McI$ be an \emph{inverse} category.
In that situation, $\bC^\McI$ is again a category of fibrant objects with fibrations and weak equivalences defined levelwise.
A natural restriction is to \emph{Reedy fibrant} diagrams, which require a compatibility between the inverse structure of $\McI$ and the fibrations of $\bC$.
Such diagram categories were studied in detail by Radulescu-Banu in \cite{radulescu-banu} and in the context of type theory by Shulman in \cite{shu15}.

One can also consider $\McI$ to carry a class of weak equivalences and ask that the diagrams $\McI \to \bC$ under consideration preserve this class, leading to the notion of a \emph{homotopical} diagram.
Such diagrams were used extensively by Szumi{\l}o to establish an equivalence of the homotopy theories of fibration categories and (finitely) complete quasicategories \cite{szumilo:agt}.
In the context of dependent type theory, such diagrams have proven indispensable in several contexts, e.g., to construct path objects on the category of models of dependent type theory \cite{kapulkin-lumsdaine:HToTT} and in the proof of homotopy canonicity by the first-named author and Sattler.

A common requirement in dependent type theory is that the category of fibrant objects also be closed under pushforwards.
Such categories of fibrant objects are presentations of locally cartesian closed $(\infty,1)$-categories.
Combining the two themes discussed above, we arrive at the fundamental question of the present paper: under what conditions is the category of homotopical Reedy fibrant digrams $\McI \to \bC$ again closed under pushforwards inside the category of all (Reedy fibrant) diagrams?

Interestingly, two extreme cases were previously established: Shulman \cite{shu15} showed that if \emph{none} of the maps in $\McI$ are weak equivalences, then pushforwards in $\bC$ give rise to pushforwards in $\bC^\McI$; while in \cite{kl21}, the case of \emph{all} maps being weak equivalences was also resolved in the positive.
By revisiting the proof of \cite{kl21} from the setting of models of dependent type theory, we are able to identify a fairly permissive condition on $\McI$, presented in \cref{thm:htpy-psfw}.

This result has applications in a variety of areas discussed above.
In dependent type theory, it allows for constructions of made-to-order models of type theory, i.e., models satisfying specific conditions on its type of propositions.
It is also a step towards proving that suitably defined locally cartesian closed categories of fibrant objects present the same homotopy theory as locally cartesian closed quasicategories \cite{kapulkin:lccqcat,cisinski:book}.

The failure of closure of homotopical diagrams under pushforwards inside all (Reedy fibrant) diagrams is also of independent interest. In general, homotopical functors can be seen as representing the $(\infty, 1)$-functors between the $(\infty, 1)$-categories presented by the two categories with weak equivalences.
Their failure to be closed under pushforward shows that the $(\infty, 1)$-category of $(\infty, 1)$-functors is not a left exact localization of the $(\infty, 1)$-category of (1-)functors, a question of independent interest.
Several such situations are discussed in \cref{ex:no-htpy-psfw}.

The paper is organized into three short sections: in \cref{sec:inverse}, we review the background on inverse categories; in \cref{sec:pushforward}, we recall and expand on the inductive definition of pushforwards in inverse diagrams; and in \cref{sec:homotopical}, we prove our main result and discuss examples and counter-examples.
In each section, we add new assumptions on the category $\bC$, making sure that they are satisfied by type-theoretic fibration categories \cite{shu15} and, when appropriate, general categories of fibrant objects \cite{brown:aht}.

\section{Inverse Diagrams} \label{sec:inverse}
We recall some preliminaries on inverse diagrams and the inductive procedure in
which one constructs inverse diagrams.

\begin{definition}
  An \emph{inverse category} $\McI$ is a category such that there exists a grading on its objects by a degree function $\deg \colon \ob\McI \to \bN$ such that if $f \colon i \to j \in \McI$ is not an identity map, then $\deg i > \deg j$.

  For each $n \in \bN$, denote by $\McI_{< n}$ the full subcategories of $\McI$ spanned by the objects of degree strictly less than $n$ and by $\partial(i/\McI)$ the full subcategory of $i/\McI$ excluding the identity map.
\end{definition}

\begin{definition}
  Let $\McI$ be a finite inverse category, and $\bC$ be a finitely complete category.
  For each $n \in \bN$, the \emph{coskeleton} functor is defined as the right
  adjoint to the restriction along $\McI_{\leq n} \hookrightarrow \McI$
  \begin{equation*}
    \begin{tikzcd}[cramped]
      {\bC^{\McI_{<n}}} & {\bC^\McI} & \bC
      \arrow[""{name=0, anchor=center, inner sep=0}, "{\cosk_n}", shift left=1.5, from=1-1, to=1-2]
      \arrow[""{name=1, anchor=center, inner sep=0}, "{\res_n}", shift left=1.5, from=1-2, to=1-1]
      \arrow["{\ev_i}", from=1-2, to=1-3]
      \arrow["\dashv"{anchor=center, rotate=90}, draw=none, from=1, to=0]
    \end{tikzcd}
  \end{equation*}
  The \emph{matching object functor} at $i \in \McI$ is the restricted monad
  $M_i \coloneqq \ev_i\relax \cdot \cosk_n\relax \cdot \res_n\relax$ and the
  \emph{matching map $m_i \colon \ev_i\relax \to M_i$} is the unit of the
  adjunction $\res_n \dashv \cosk_n$ restricted along the map evaluating at $i$.
\end{definition}

Explicitly, for each $X \in \bC^\McI$, the matching object is the limit:
\begin{equation*}
  M_iX = \lim(\partial(i/\McI) \to \McI \xrightarrow{X} \bC)
\end{equation*}
and the matching map $m_iX \colon X_i \to M_iX$ is the unique map induced by the
cone $(X_f \colon X_i \to X_j)_{f \colon i \to j \in \partial(i/\McI)}$.

\begin{definition}
  Given a map $f \colon X \to Y \in \bC^\McI$ of diagrams, the \emph{relative
    matching map} at $i \in \McI$ is the comparison map between the pullback as
  follows.
  \begin{equation*}
    \begin{tikzcd}[cramped]
      {Y_i} \\
      & {X_i \times_{M_iX} M_iY} & {M_iY} \\
      & {X_i} & {M_iX}
      \arrow["{\hM_if}"{description}, dashed, from=1-1, to=2-2]
      \arrow["{m_iY}", curve={height=-12pt}, from=1-1, to=2-3]
      \arrow["{f_i}"', curve={height=12pt}, from=1-1, to=3-2]
      \arrow[from=2-2, to=2-3]
      \arrow[from=2-2, to=3-2]
      \arrow["\lrcorner"{anchor=center, pos=0.05}, draw=none, from=2-2, to=3-3]
      \arrow["{M_if}", from=2-3, to=3-3]
      \arrow["{m_iX}"', from=3-2, to=3-3]
    \end{tikzcd}
  \end{equation*}

  If $\bC$ is equipped with a wide subcategory of fibrations
  $\McF \subseteq \bC$, a \emph{Reedy fibration} in $\bC^\McI$ is a map of
  diagrams $p \colon E \twoheadrightarrow B \in \bC^\McI$ where each relative
  matching map
  $\hM_ip \colon E_i \twoheadrightarrow B_i \times_{M_iB} M_iE \in \McF$ is a
  fibration in $\bC$.

  We say a diagram $X \in \bC^\McI$ is \emph{Reedy fibrant} when the map
  $X \to 1 \in \bC^\McI$ is a Reedy fibration.
  This is the same as saying each matching map $m_iX \colon X_i \to M_iX$ is a
  fibration.
\end{definition}

\section{Pushforwards in Inverse Diagrams} \label{sec:pushforward}
We first recall and expand on the inductive procedure for
constructing pushfowards in inverse diagram categories given by
\cite{fkl24}.

For this, we first provide an alternative calculation of the matching object.
\begin{lemma}\label{lem:mat-obj-lim}
  Let $p \colon E \to B \in \bC^\McI$ be a map of inverse diagrams and
  $i \in \McI$.
  The map $M_ip \colon M_iE \to M_iB$, viewed as an object in
  $\sfrac{\bC}{M_iB}$, is the limit of the $\partial(i/\McI)$-shaped diagram $D$
  valued in $\sfrac{\bC}{M_iB}$ taking each
  $f \colon i \to j \in \partial(i/\McI)$ to
  $D_f \coloneqq \proj_f^*E_j \to M_iB \in \sfrac{\bC}{M_iB}$ with action on a
  map $g \colon j \to j'$ defined by universality of the pullback.
  \begin{equation*}
    \begin{tikzcd}[cramped, row sep=small, column sep=small]
      & {M_iE} \\
      {D_f} &&& {E_j} \\
      && {D_{gf}} && {E_{j'}} \\
      &&& {B_j} \\
      & {M_iB} &&& {B_{j'}}
      \arrow[from=1-2, to=2-1]
      \arrow[from=2-1, to=2-4]
      \arrow["{D_g}"{description, pos=0.7}, from=2-1, to=3-3]
      \arrow[from=2-1, to=5-2]
      \arrow["{E_g}", from=2-4, to=3-5]
      \arrow["{p_j}"{pos=0.7}, from=2-4, to=4-4]
      \arrow[from=3-3, to=5-2]
      \arrow["{p_{j'}}", from=3-5, to=5-5]
      \arrow["{B_g}", from=4-4, to=5-5]
      \arrow["{\proj_f}"{description}, from=5-2, to=4-4]
      \arrow["{\proj_{gf}}"', from=5-2, to=5-5]
      \arrow["{\proj_{gf}}"'{pos=0.6}, crossing over, from=1-2, to=3-5]
      \arrow["{\proj_f}", from=1-2, to=2-4]
      \arrow[crossing over, from=1-2, to=3-3]
      \arrow["{M_ip}"{description, pos=0.7}, crossing over, curve={height=6pt}, from=1-2, to=5-2]
      \arrow[crossing over, from=3-3, to=3-5]
    \end{tikzcd}
  \end{equation*}
\end{lemma}
\begin{proof}
  Functoriality of the weighted limit defines a cone
  $(M_iE \to D_f \in \sfrac{\bC}{M_iB})_{\id\relax \neq f \in i/\McI}$.
  To check universality of this cone, take another cone
  $(X \to D_f \in \sfrac{\bC}{M_iB})_{\id\relax \neq f \in i/\McI}$.
  Such a cone is exactly a map $X \to M_iB$ along with a family of maps
  $(k_f \colon X \to E_j)_{\id\relax \neq f \in i/\McI}$ such that for each
  $\id \neq f \colon i \to j$, one has $\proj_f \cdot x = p_j \cdot k_f$, and
  for each $g \colon j \to j'$ under $i$, one has $E_g \cdot k_f = k_{gf}$.
  By the universality of the weighted limit, this induces uniquely a map
  $X \to M_iE$ factoring $x \colon X \to M_iB$.
\end{proof}

Then, the inductive procedure for constructing pushforwards in inverse diagram
categories from \cite{fkl24} can be rephrased as follows.
\begin{lemma}[{\cite[Corollary 5.6]{fkl24}}]\label{lem:inv-psfw}
  Fix $\McI$ an inverse category and a finitely completely category $\bC$.
  Let there be a map of diagrams $p \colon B \to A \in \bC^\McI$ along with an
  object $k \colon C \to B \in \sfrac{\bC^\McI}{B}$.
  Assume that for each $i \in \McI$
  \begin{itemize}
    \item the pushforward of $C_i \to B_i \in \sfrac{\bC}{B_i}$ along
    $p_i \colon B_i \to A_i \in \bC$ exists; and
    \item the pushforward of $M_iC \to M_iB \in \sfrac{\bC}{M_iB}$ along
    $M_ip \colon M_iB \to M_iA \in \bC$ exists.
  \end{itemize}

  Then, the pushforward $p_*C$ exists.
  The component of this pushforward at each $i \in \McI$ is equipped with a
  projection map $\kappa_i \colon (p_*C)_i \to (p_i)_*C_i$ and obtained as the
  following pullback over $A_i$ in which the bottom and right maps are induced
  uniquely by the universality of the limit indexed by the strictly
  degree-decreasing decreasing maps $f \colon i \to j$
  \begin{equation*}
    \begin{tikzcd}[cramped, column sep=small]
      {(p_*C)_i} && {(p_i)_*C_i} \\
      {\displaystyle{\lim_{\id\relax \neq f \colon i \to j}^{\bC/A_i}(A_f^* (p_*C)_j)}} & {} & {\displaystyle{\lim_{\id\relax \neq f \colon i \to j}^{\bC/A_i}{((p_i)_*B_f^*C_j)}}} \\
      {A_f^*(p_*C)_j} & {A_f^*(p_j)_*C_j} & {(p_i)_*B_f^*C_j}
      \arrow["{\kappa_i}", from=1-1, to=1-3]
      \arrow[from=1-1, to=2-1]
      \arrow["\lrcorner"{anchor=center, pos=0}, draw=none, from=1-1, to=2-2]
      \arrow["{((p_*k)_i, (p_*C)_f)}"', curve={height=80pt}, from=1-1, to=3-1]
      \arrow[from=1-3, to=2-3]
      \arrow["{(p_i)_*(k_i,C_f)}", curve={height=-80pt}, from=1-3, to=3-3]
      \arrow[from=2-1, to=2-3]
      \arrow["{\proj_f}", from=2-1, to=3-1]
      \arrow["{\proj_f}"', from=2-3, to=3-3]
      \arrow["{A_f^*\kappa_j}"', from=3-1, to=3-2]
      \arrow["{(B_f^*\ev)^\dagger}"', from=3-2, to=3-3]
    \end{tikzcd}
  \end{equation*}
  and the maps $(B_f^*\ev)^\dagger$ and $(p_i)_*(k_i,C_f)$ are respectively
  the comparison maps induced by the universality of the pushforward.
  \begin{equation*}
    \begin{tikzcd}[cramped, row sep=small]
      {C_i} &&& {(p_i)_*C_i} \\
      {B_f^*C_j} && {A_f^*(p_j)_*C_j} & {(p_i)_*B_f^*C_j} \\
      & {C_j} && {(p_j)_*C_j} \\
      {B_i} && {A_i} \\
      & {B_j} && {A_j}
      \arrow["{(k_i,C_f)}"', from=1-1, to=2-1]
      \arrow[from=1-1, to=3-2]
      \arrow["{(p_i)_*(k_i,C_f)}", from=1-4, to=2-4]
      \arrow[from=2-1, to=3-2]
      \arrow[from=2-1, to=4-1]
      \arrow["{(B_f^*\ev)^\dagger}", from=2-3, to=2-4]
      \arrow[from=2-3, to=4-3]
      \arrow[from=2-4, to=4-3]
      \arrow[from=3-4, to=5-4]
      \arrow["{p_i}"{pos=0.7}, from=4-1, to=4-3]
      \arrow["{B_f}"', from=4-1, to=5-2]
      \arrow["{A_f}", from=4-3, to=5-4]
      \arrow["{p_j}"', from=5-2, to=5-4]
      \arrow[from=2-3, to=3-4, crossing over]
      \arrow[from=3-2, to=5-2, crossing over]
    \end{tikzcd}
  \end{equation*}
\end{lemma}
\begin{proof}
  By \cite[Construction 2.13 and Corollary 5.6]{fkl24}, each $\kappa_i$ is
  isomorphic to the following pullback over $A_i$
  \begin{equation*}
    \begin{tikzcd}[cramped]
      {(p_*C)_i} & {(p_i)_*C_i} \\
      {(m_iA)^*M_i(p_*C)} & {(p_i)_*(m_iB)^*(M_iC)}
      \arrow["{\kappa_i}", from=1-1, to=1-2]
      \arrow[from=1-1, to=2-1]
      \arrow["\lrcorner"{anchor=center, pos=0.1}, draw=none, from=1-1, to=2-2]
      \arrow["{(p_i)_*(\widehat{M}_ik)}", from=1-2, to=2-2]
      \arrow[from=2-1, to=2-2]
    \end{tikzcd} \in \sfrac{\bC}{A_i}
  \end{equation*}
  where the exponential transpose of the bottom map is
  \begin{equation*}
    (m_iB)^*((M_ip)^*M_i(p_*C)) \cong
    (m_iB)^*M_i(p^*p_*C) \xrightarrow{(m_iB)^*M_i(\ev)}
    (m_iB)^*M_iC \in \sfrac{\bC}{B_i}
  \end{equation*}
  By \Cref{lem:mat-obj-lim}, we see that as objects and maps in $\sfrac{\bC}{M_iB}$,
  \begin{equation*}
    \begin{tikzcd}[cramped, column sep=7em]
      {M_i(p^*p_*C)} & {M_iC} \\
      {\displaystyle \lim_{\id \neq f \colon i \to j}^{\sfrac{\bC}{M_iB}} \proj_f^*(p_j^*(p_*C)_j)} & {\displaystyle \lim_{\id \neq f \colon i \to j}^{\sfrac{\bC}{M_iB}} \proj_f^*(C_j)}
      \arrow["{M_i(\ev)}", from=1-1, to=1-2]
      \arrow["\cong"', tail reversed, from=1-1, to=2-1]
      \arrow["\cong", tail reversed, from=1-2, to=2-2]
      \arrow["{\scriptstyle \lim_{\id \neq f \colon i \to j}^{\bC/M_iB} \proj_f^*\ev_j}"', from=2-1, to=2-2]
    \end{tikzcd}
  \end{equation*}
  Thus, further pulling back along $m_iB \colon B_i \to M_iB$, one observes
  \begin{equation*}
    \begin{tikzcd}[cramped, column sep=7em]
      {(m_iB)^*M_i(p^*p_*C)} & {(m_iB)^*M_iC} \\
      {\displaystyle \lim_{\id \neq f \colon i \to j}^{\bC/B_i} B_f^*(p_j^*(p_*C)_j)} & {\displaystyle \lim_{\id \neq f \colon i \to j}^{\bC/B_i} B_f^*C_j}
      \arrow["{(m_iB)^*M_i(\ev)}", from=1-1, to=1-2]
      \arrow["\cong"', tail reversed, from=1-1, to=2-1]
      \arrow["\cong", tail reversed, from=1-2, to=2-2]
      \arrow["{\scriptstyle \lim_{\id \neq f \colon i \to j}^{\bC/B_i} B_f^*\ev_j}"', from=2-1, to=2-2]
    \end{tikzcd}
  \end{equation*}
  Meanwhile, for each $f \colon i \to j \neq \id$, the exponential transpose of
  the map
  ${A_f^*(p_*C)_j} \xrightarrow{A_f^*\kappa_j} {A_f^*(p_j)_*C_j}
  \xrightarrow{(B^*_f\ev)^\dagger} {(p_i)_*B_f^*C_j}$ is
  \begin{equation*}
    \begin{tikzcd}[cramped]
      {B_f^*p_j^*(p_*C)_j} & {B_f^*p_j^*(p_j)_*C_j} & {B_f^*C_j}
      \arrow["{B_f^*p_j^*\kappa_j}", from=1-1, to=1-2]
      \arrow["{B_f^*\ev}", from=1-2, to=1-3]
    \end{tikzcd} \in \sfrac{\bC}{B_i}
  \end{equation*}
  We now conclude by noting that by \cite[Constructions 2.14 and 2.16]{fkl24},
  the counit $\ev_j \colon p_j^*(p_*C)_j \to C_j \in \sfrac{\bC}{A_j}$ at
  component $j$ is implemented as the exponential transpose of the map
  $p_j^*\kappa_j \colon (p_*C)_j \to (p_j)_*C_j$.
\end{proof}

Next, assume that $\bC$ is equipped with a pullback-stable wide subcategory of
fibrations $\McF \subseteq \bC$ such that pushforwards of fibrations along
fibrations exist.
The goal is to now prove that in $\bC^\McI$, the pushforward of a Reedy
fibration along a Reedy fibration exist and remains a Reedy fibration.

Existence is given by \Cref{lem:inv-psfw}.
To help with showing Reedy fibrancy, we first prove the following distributive
law.
\begin{lemma}[{\cite[Paragraph 1.2]{gk13}}]\label{lem:psfw-distr}
  Suppose that one has a map $p \colon B \to A \in \bC$ and
  $k \colon C \to B \in \sfrac{\bC}{B}$ such that the pushforward
  $p_*k \colon p_*C \to A \in \sfrac{\bC}{A}$ exists.

  If $d \colon D \to C$ is such that the pushforward
  $q_*(\ev^*d) \colon q_*(\ev^*D) \to p_*C$ of
  $\ev^*d \colon \ev^*D \to p^*p_*C$ along the connecting map
  $q \colon p^*p_*C \to p_*C$ of the pullback of $p \colon B \to A$ along
  $p_*k \colon p_*C \to A$ also exists, then the composition
  $q_*(\ev^*D) \to p_*C \to A$ is also the pushforward of $D \to C \to B$ along
  $p$.
  \begin{equation*}
    \begin{tikzcd}[cramped]
      D & {\ev^*D} && {q_*(\ev^*D)} \\
      C & {p^*p_*C} && {p_*C} \\
      & B && A
      \arrow["d"', from=1-1, to=2-1]
      \arrow[from=1-2, to=1-1]
      \arrow["\lrcorner"{anchor=center, pos=0.125, rotate=-90}, draw=none, from=1-2, to=2-1]
      \arrow["{\ev^*d}", from=1-2, to=2-2]
      \arrow["{q_*(\ev^*d)}", from=1-4, to=2-4]
      \arrow["k"', from=2-1, to=3-2]
      \arrow["\ev"', from=2-2, to=2-1]
      \arrow["q", from=2-2, to=2-4]
      \arrow[from=2-2, to=3-2]
      \arrow["\lrcorner"{anchor=center, pos=0.05}, draw=none, from=2-2, to=3-4]
      \arrow["{p_*k}", from=2-4, to=3-4]
      \arrow["p"', from=3-2, to=3-4]
    \end{tikzcd}
  \end{equation*}
\end{lemma}
\begin{proof}
  Fix $x \colon X \to A \in \sfrac{\bC}{A}$ so that we must exhibit a natural bijection
  \begin{equation*}
    \sfrac{\bC}{A}(x, (p_*k)_!(q_*\ev^*d)) \cong \sfrac{\bC}{B}(p^*x, k_!d)
  \end{equation*}
  A map $X \to q_*(\ev^*D)$ over $A$ is exactly a choice of a map
  $u \colon X \to p_*C$ over $A$ along with a factorisation of the chosen map
  $u$ via $q_*(\ev^*d)$.
  Applying the same reasoning to a map $p^*X \to D$ over $B$, we obtain the
  following bijections.
  \begin{align*}
    \sfrac{\bC}{A}(x, (p_*k)_!(q_*\ev^*d))
    &\cong \coprod_{u \in\sfrac{\bC}{A}(x, p_*k)} \sfrac{\bC}{p_*C}(u,q_*(\ev^*d))
    \\
    \sfrac{\bC}{B}(p^*x, k_!d)
    &\cong \coprod_{u \in\sfrac{\bC}{B}(p^*x, k)} \sfrac{\bC}{C}(u,d)
  \end{align*}

  We simplify the first bijection.
  Given $u \colon X \to p_*C$, one has
  $\sfrac{\bC}{p_*C}(u, q_*(\ev^*d)) \cong \sfrac{\bC}{p^*q_*C}(q^*\allowbreak u, \ev^*d)
  \cong \sfrac{\bC}{C}(\ev_!q^*u, \allowbreak d)$.
  So,
  $\sfrac{\bC}{A}(x, (p_*k)_!(q_*\ev^*d)) \cong
  \coprod_{u \in\sfrac{\bC}{A}(x, p_*k)} \sfrac{\bC}{C}(\ev_!q^*u, \allowbreak
  d)$, and it suffices to show a bijection
  \begin{equation*}
    \coprod_{u \in\sfrac{\bC}{A}(x, p_*k)} \sfrac{\bC}{C}(\ev_!q^*u, d)
    \cong
    \coprod_{u \in\sfrac{\bC}{B}(p^*x, k)} \sfrac{\bC}{C}(u,d)
  \end{equation*}
  But $\sfrac{\bC}{A}(x,p_*k) \cong \sfrac{\bC}{B}(p^*x,k)$ by the exponential
  transpose $(-)^\dagger$, and for each $u \colon X \to p_*C$, one has
  $\ev_!q^*u = \ev_!p^*u = u^\dagger$, so
  $(u \in \sfrac{\bC}{A}(x,p_*k), f \in \sfrac{\bC}{C}(\ev_!q^*u,d)) \mapsto
  (u^\dagger \in \sfrac{\bC}{B}(p^*x,k), f \in \sfrac{\bC}{C}(u^\dagger, d))$
  gives the required bijection.
\end{proof}

We then reproduce the usual result from model category theory that the matching
object functor preserves fibrations.
This is needed because the pushforward formula from \Cref{lem:inv-psfw} requires
pushforwards along the map between matching objects to exist.
\begin{proposition}\label{prop:mat-fib}
  Suppose each underslice of $\McI$ is finite.
  Then matching object functors preserve fibrations and Reedy fibrations are
  pointwise fibrations.
\end{proposition}
\begin{proof}
  Let $p \colon E \twoheadrightarrow B \in \bC^\McI$ be a Reedy fibration so
  that we must show $M_ip \colon M_iE \to M_iB \in \bC$ is a fibration for each
  $i \in \McI$.

  Fix $i \in \McI$.
  To show that $M_ip \colon M_iE \to M_iB \in \bC$ is a fibration is to show
  that it is a fibrant object of $\sfrac{\bC}{M_iB}$.
  By \cite[Lemma 11.8]{shu15}, we are done if we just show that $M_iE$ is the
  limit of a Reedy fibrant diagram in $\sfrac{\bC}{M_iB}$.
  \Cref{lem:mat-obj-lim} already expresses $M_ip \colon M_iE \to M_iB$ as the
  limit the $\partial(i/\McI)$-shaped diagram $D$ valued in $\sfrac{\bC}{M_iB}$
  taking each $\id \neq f \colon i \to j$ to $D_f \coloneqq \proj_f^*E_j \to M_iB$.
  We just need to show $D \in (\sfrac{\bC}{M_iB})^{\partial(i/\McI)}$ is Reedy fibrant.

  To do this, fix $f \colon i \to j \in \partial(i/\McI)$.
  Then,
  \begin{align*}
    M_fD
    \cong \lim_{\id \neq g \colon j \to j'} D_{gf}
    \cong \lim_{\id \neq g \colon j \to j'} \proj_{gf}^*E_{j'}
    \cong \proj_{f}^*(\lim_{\id \neq g \colon j \to j'} B_g^*E_{j'})
  \end{align*}
  Because $D_f = \proj_f^*E_j$, this means that the matching map
  $m_f \colon D_f \to M_fD$ is the pullback of
  $E_j \to \lim_{\id \neq g \colon j \to j'} B_g^*E_{j'} \in \sfrac{\bC}{B_j}$
  along $\proj_f \colon M_iB \to B_j$.
  But this is exactly the relative matching map of the Reedy fibration
  $p \colon E \twoheadrightarrow B \in \bC^\McI$ at $j$, so the result follows.
\end{proof}

Now, we can show pushforwards along Reedy fibrations preserve Reedy fibrations.
\begin{proposition}\label{prop:reedy-fib-psfw}
  Suppose $\bC$ is equipped with a pullback-stable wide subcategory of
  fibrations such that fibrations are stable under pushfoward along fibrations
  and $\McI$ is finite.
  Then, Reedy fibrations in $\bC^\McI$ are stable under pushforwards along Reedy
  fibrations.
\end{proposition}
\begin{proof}
  By \Cref{prop:mat-fib}, if $p \colon B \twoheadrightarrow A \in \bC^\McI$
  is a Reedy fibration then so is each of the maps
  $M_ip \colon M_iB \twoheadrightarrow M_iA \in \bC$ for $i \in \McI$.
  Therefore, by \Cref{lem:inv-psfw}, the pushforwards of a Reedy fibration
  $k \colon C \twoheadrightarrow B$ along $p \colon B \twoheadrightarrow A$
  exists and each relative matching map $\hM_i(p_*k)$ is a pullback of
  $(p_i)_*\hM_ik$.
  By \Cref{lem:psfw-distr}, $(p_i)_*\hM_ik$ is the pushforward of $\ev^*\hM_ik$
  along $((p_i)_*k)^*p_i$.
  The result now follows because $k$ is a Reedy fibration.
\end{proof}

\section{Pushforwards in Homotopical Inverse Diagrams} \label{sec:homotopical}

Next, let $\McI$ and $\bC$ be equipped with two wide subcategories of weak
equivalences respectively containing the isomorphisms and closed under
2-out-of-3.
\begin{definition}
  Denote by $\bC^\McI_\Msh \hookrightarrow \bC^\McI$ the full subcategory of
  diagrams in $\bC^\McI$ preserving weak equivalences.
\end{definition}

We wish for $\bC^\McI_\Msh$ to admit certain classes of pushforwards.
Namely, we wish for homotopical Reedy fibrations to be closed under pushforwards
of homotopical Reedy fibrations.
For this to happen, we assume some logical conditions on the behaviour of
pullbacks and pushforwards in $\bC$ with respect to the chosen class of weak
equivalences and fibrations.
\begin{assumption}\label{asm:pb-pi-well-behaved}
  We say that pullbacks in $\bC$ are \emph{homotopically logically behaved}
  with respect to the chosen class of fibrations and weak equivalences when
  \begin{itemize}
    \item weak equivalences are preserved by pullback along fibrations; and
    \item pullbacks preserve weak equivalences between fibrations.
  \end{itemize}

  Similarly, pushforwards in $\bC$ are \emph{homotopically logically behaved}
  with respect to the chosen class of fibrations and weak equivalences when
  \begin{itemize}
    \item pushforwards along fibrations preserve weak equivalences between fibrations; and
    \item whenever $g,f$ below are weak equivalences and $p,p'$ are fibrations,
    the ``precomposition'' map $(g^*\ev)^\dagger \colon f^*p_*C \to p'_*g^*C$ is
    also a weak equivalence.
    \begin{equation*}
      \begin{tikzcd}[cramped, row sep=small]
        {g^*C} && {f^*p_*C} & {p'_*g^*C} \\
        & C && {p_*C} \\
        {B'} && {A'} \\
        & B && A
        \arrow["\sim", from=1-1, to=2-2]
        \arrow[two heads, from=1-1, to=3-1]
        \arrow["{(g^*\ev)^\dagger}", "\sim"', from=1-3, to=1-4]
        \arrow[two heads, from=1-3, to=3-3]
        \arrow[two heads, from=1-4, to=3-3]
        \arrow[two heads, from=2-4, to=4-4]
        \arrow["{p'}"{pos=0.7}, two heads, from=3-1, to=3-3]
        \arrow["g"', "\sim", from=3-1, to=4-2]
        \arrow["f", "\sim"', from=3-3, to=4-4]
        \arrow["p"', two heads, from=4-2, to=4-4]
        \arrow[crossing over, two heads, from=2-2, to=4-2]
        \arrow["\sim", crossing over, from=1-3, to=2-4]
      \end{tikzcd}
    \end{equation*}
  \end{itemize}
\end{assumption}

Now, the task is to show that if
$p \colon B \twoheadrightarrow A \in \bC_\Msh^\McI$ is a map between homotopical
diagrams, then the pushforwards of any homotopical
$k \colon C \to B \in \sfrac{\bC_\Msh^\McI}{B}$ remains homotopical.
This amounts to showing that whenever
$w \colon i \xrightarrow{\sim} j \in \McI$, then
$(p_*C)_i \to A_f^*(p_*C)_j \in \bC$ is a weak equivalence.
By \Cref{lem:inv-psfw} and the logical homotopical behaviour of pushfowards,
we are done if we somehow have that $(p_*C)_i \to A_f^*(p_*C)_j \in \bC$ is a
pullback of $(p_i)_*C_i \to (p_i)_*B_f^*C_j$.
This is the same as requiring $w$ to be initial in the boundary of $i$.

Another option is to follow \cite[Proposition 5.13]{kl21} and show that all
horizontal maps and the right-most map in \Cref{lem:inv-psfw} are weak
equivalences, so that we obtain the required result by 2-out-of-3.
This method amounts to showing that the comparison maps
$\kappa_i \colon (p_*C)_i \to (p_i)_*C_i$ from \Cref{lem:inv-psfw} are weak
equivalences by induction.
For this, we need to know that the map between limits which $\kappa_i$ occur as
is in fact a map between limits of Reedy fibrant diagrams in the slice
$\sfrac{\bC}{A_i}$.
The following computation is useful for verifying the Reedy fibrancy.
\begin{lemma}\label{lem:reedy-fib-local}
  Let $p \colon E \to B \in \bC^\McI$ be a Reedy fibration and $i \in \McI$.
  The diagram $\ol{E} \colon \partial(i/\McI) \to \sfrac{\bC}{B_i}$ with action
  on objects $f \colon i \to j \in \partial(i/\McI)$ defined as
  $\ol{E}_f \coloneqq B_f^*E_j \to B_i \in \sfrac{\bC}{B_i}$ and action on maps
  $g \colon j \to j'$ defined by universality of the pullback is a Reedy fibrant
  diagram.
  \begin{equation*}
    \begin{tikzcd}[cramped, row sep=small, column sep=small]
      {\ol{E}_f} &&& {E_j} \\
      && {\ol{E}_{gf}} && {E_{j'}} \\
      &&& {B_j} \\
      & {B_i} &&& {B_{j'}}
      \arrow[from=1-1, to=1-4]
      \arrow["{\ol{E}_g}"{description}, from=1-1, to=2-3]
      \arrow[from=1-1, to=4-2]
      \arrow["{E_g}", from=1-4, to=2-5]
      \arrow["{p_j}"{pos=0.7}, from=1-4, to=3-4]
      \arrow[from=2-3, to=4-2]
      \arrow["{p_{j'}}", from=2-5, to=4-5]
      \arrow["{B_g}", from=3-4, to=4-5]
      \arrow["{B_f}"{description}, from=4-2, to=3-4]
      \arrow["{B_{gf}}"', from=4-2, to=4-5]
      \arrow[crossing over, from=2-3, to=2-5]
    \end{tikzcd}
  \end{equation*}
\end{lemma}
\begin{proof}
  Fix $f \colon i \to j \in \partial(i/\McI)$.
  We must show that the matching map
  $m_f\ol{E} \colon \ol{E}_f \to M_f\ol{E} \in \sfrac{\bC}{B_i}$ is a fibration.
  First, we calculate the matching object
  $M_f\ol{E} \to B_i \in \sfrac{\bC}{B_i}$.
  Because $\sfrac{f}{i/\McI} \cong j/\McI$, we see that
  \begin{align*}
    M_f\ol{E} = \lim_{\id \neq g \colon f \to f' \in i/\McI} \ol{E}_{f'}
    \cong  \lim_{\id \neq g \colon j \to j'} B_{gf}^*E_{j'}
    \cong B_f^*(\lim_{\id \neq g \colon j \to j'} B_g^*E_{j'})
  \end{align*}
  Then, we recall that $\ol{E}_f = B_f^*(E_j)$ by definition.
  Because fibrations are pullback-stable, it suffices to show that
  $E_j \to \lim_{\id \neq g \colon j \to j'} B_g^*E_{j'} \in \sfrac{\bC}{B_j}$
  is a fibration.
  By \Cref{lem:mat-obj-lim}, one has
  $M_jE \cong \lim_{\id \neq g \colon j \to j'}\proj_g^*E_{j'} \in
  \sfrac{\bC}{M_jB}$.
  Because each $B_g = \proj_g \cdot m_jB$, it follows that
  $\lim_{\id \neq g \colon j \to j'} B_g^*E_{j'} \cong (m_jB)^* \lim_{\id \neq g
    \colon j \to j'} \proj_g^*E_{j'} \cong (m_jB)^*M_jB$.
  By definition of $E \twoheadrightarrow B \in \bC^\McI$ being a Reedy
  fibration, it follows now that $E_j \twoheadrightarrow (m_jB)^*M_jB$ is a
  fibration.
\end{proof}

\begin{lemma}\label{lem:htpy-psfw-all-we}
  Suppose that $i \in \McI$ is such that all maps under $i$ are weak
  equivalences.
  Then, for a Reedy fibration between homotopical diagrams
  $p \colon B \twoheadrightarrow A \in \bC^\McI_\Msh$ and a Reedy fibrant
  homotopical object
  $k \colon C \twoheadrightarrow B \in \sfrac{\bC^\McI_\Msh}{B}$, the
  pushforward $p_*C \twoheadrightarrow A \in \sfrac{\bC^\McI}{B}$ sends all maps
  under $i$ to weak equivalences.
\end{lemma}
\begin{proof}
  We first show that each comparison
  map $\kappa_j \colon (p_*C)_j \to (p_j)_*C_j$ is a weak equivalence.
  Inductively assume that this is true for all $j$ with degree less than $i$.
  Because $C \twoheadrightarrow B \in \sfrac{\bC^\McI}{B}$ is Reedy fibrant and
  pushforwards along fibrations preserve fibrations by \Cref{lem:psfw-distr}, it
  suffices to show that the map between limits
  \begin{equation*}
    \lim_{\id\relax \neq f \colon i \to j}^{\bC/A_i}(A_f^* (p_*C)_j)
    \xrightarrow{\lim_{\id\relax \neq f \colon i \to j}^{\bC/A_i}((B_f^*\ev)^\dagger \cdot A_f^*\kappa_j)}
    \lim_{\id\relax \neq f \colon i \to j}^{\bC/A_i}{((p_i)_*B_f^*C_j)}
  \end{equation*}
  is a weak equivalence.

  By induction and right properness, each $A_f^*\kappa_j$ is a weak equivalence.
  Further, pushforwards are logically behaved, so each ``precomposition map''
  $(B_f^*\ev)^\dagger$ is a weak equivalence because all maps under $i$ are weak
  equivalences and the diagrams $A,B \in \bC^\McI$ are homotopical by
  assumption.
  Thus, by \cite[Lemma 11.8]{shu15}, the induced map between limits is a weak
  equivalence if we show that the limits are limits of Reedy fibrant diagrams.
  Because $p_*C \twoheadrightarrow A$ is Reedy fibrant by
  \Cref{prop:reedy-fib-psfw}, it follows by \Cref{lem:reedy-fib-local} that the
  diagram $(A_f^*(p_*C) \in \sfrac{\bC}{A_i})_{\id \neq f \colon i \to j}$ is
  Reedy fibrant.
  By \Cref{lem:reedy-fib-local} again,
  $(B_f^*C_j \in \sfrac{\bC}{B_i})_{\id \neq f \colon i \to j}$ is Reedy
  fibrant, and because the pushforward map is continuous, it follows that
  $((p_i)_*B_f^*C_j \in \sfrac{\bC}{A_i})_{\id \neq f \colon i \to j}$ is also
  Reedy fibrant.
  This completes the inductive argument to show each $\kappa_i$ is a weak
  equivalence.

  Now let $f \colon i \to j$ be a non-identity map.
  Then, by 2-out-of-3, the map $(k_i, C_f) \colon C_i \to B_f^*C_j$ is a weak
  equivalence.
  Therefore, by assumption, $(p_i)_*(k_i,C_f)$ is a weak equivalence.
  We have already showed that $\kappa_i$ and $\kappa_j$ is a weak equivalence,
  so in particular $A_f^*\kappa_j$ is a weak equivalence.
  Also, we have already observed that $(B_f^*\ev)^\dagger$ is a weak
  equivalence.
  Therefore, by 2-out-of-3, it follows that
  $((p_*k)_i, (p_*C)_f) \colon (p_*C)_i \to A_f^*(p_*C)_j$ is a weak
  equivalence.
  \begin{equation*}
    \begin{tikzcd}[cramped]
      {(p_*C)_i} && {(p_i)_*C_i} \\
      {A_f^*(p_*C)_j} & {A_f^*(p_j)_*C_j} & {(p_i)_*B_f^*C_j} \\
      {A_i} & {(p_*C)_j} \\
      & {A_j}
      \arrow["{\kappa_i}", "\sim"', from=1-1, to=1-3]
      \arrow["{((p_*k)_i, (p_*C)_f)}"', "\sim", dashed, from=1-1, to=2-1]
      \arrow["{(p_i)_*(k_i,C_f)}", "\sim"', from=1-3, to=2-3]
      \arrow["{A_f^*\kappa_j}"', "\sim", from=2-1, to=2-2]
      \arrow[two heads, from=2-1, to=3-1]
      \arrow["\sim", from=2-1, to=3-2]
      \arrow["{(B_f^*\ev)^\dagger}"', "\sim", from=2-2, to=2-3]
      \arrow["{A_f}"', "\sim", from=3-1, to=4-2]
      \arrow[two heads, from=3-2, to=4-2]
    \end{tikzcd}
  \end{equation*}
  Applying 2-out-of-3 again and using the fact that
  $A_f^*(p_*C)_j \xrightarrow{\sim} (p_*C)_j$ is a weak equivalence because
  $A_f$ is a weak equivalence, it follows that
  $(p_*C)_f \colon (p_*C)_i \to (p_*C)_j$ is a weak equivalence.
\end{proof}

\begin{theorem}\label{thm:htpy-psfw}
  Suppose that for each $i \in \McI$, if there is a weak equivalence
  $\id \neq w \colon i \to j$ coming out of $i$ then either all maps coming out
  of $i$ are weak equivalences or $w$ is the initial object in
  $\partial(i/\McI)$.

  Then, pushforwards of homotopical Reedy fibrations along homotopical Reedy
  fibrations are again homotopical Reedy fibrations.
\end{theorem}
\begin{proof}
  Fix a Reedy fibration between homotopical diagrams
  $p \colon B \twoheadrightarrow A \in \bC^\McI_\Msh$ and a Reedy fibrant
  homotopical object
  $k \colon C \twoheadrightarrow B \in \sfrac{\bC^\McI_\Msh}{B}$.
  Existence and fibrancy of the pushforward $(p_*C) \to A$ is given by
  \Cref{prop:reedy-fib-psfw}.
  It remains to check that $p_*C$ is homotopical.

  Let $w \colon i \to j \in \partial(i/\McI)$ be a non-identity weak equivalence.
  The first case where all objects and maps in $i/\McI$ are weak equivalence is
  covered by \Cref{lem:htpy-psfw-all-we}.
  In the second case where $w \in \partial(i/\McI)$ is initial,
  \Cref{lem:inv-psfw} shows that $(p_*C)_i \to A_f^*(p_*C)_j$ is a pullback of
  $(p_i)_*C_i \to (p_i)_*(k_i,C_w)$, which is a trivial fibration, and is
  therefore a weak equivalence.
  \begin{equation*}
    \begin{tikzcd}[cramped]
      {(p_*C)_i} && {(p_i)_*C_i} \\
      {A_w^*(p_*C)_j} & {A_w^*(p_j)_*C_j} & {(p_i)_*B_w^*C_j} \\
      {A_i} & {(p_*C)_j} \\
      & {A_j}
      \arrow["{\kappa_i}", from=1-1, to=1-3]
      \arrow["{((p_*k)_i, (p_*C)_f)}"', "\sim", dashed, two heads, from=1-1, to=2-1]
      \arrow["\lrcorner"{anchor=center, pos=0.05}, draw=none, from=1-1, to=2-2]
      \arrow["{(p_i)_*(k_i,C_f)}", "\sim"', two heads, from=1-3, to=2-3]
      \arrow["{A_f^*\kappa_j}", from=2-1, to=2-2]
      \arrow[two heads, from=2-1, to=3-1]
      \arrow["\sim", from=2-1, to=3-2]
      \arrow["{(B_w^*\ev)^\dagger}", from=2-2, to=2-3]
      \arrow["{A_f}"', "\sim", from=3-1, to=4-2]
      \arrow[two heads, from=3-2, to=4-2]
    \end{tikzcd}
  \end{equation*}
  Thus, by 2-out-of-3, $(p_*C)_w \colon (p_*C)_i \to (p_*C)_j$ is also a weak
  equivalence.
\end{proof}

We now observe a few examples and counter-examples covered by \Cref{thm:htpy-psfw}.
\begin{example}
  In addition to \Cref{thm:htpy-psfw} applying to the case where all maps in
  $\McI$ are weak equivalences, we can also take the following shapes for
  $(\McI,\McW_\McI)$.
  \begin{center}
    \begin{minipage}{0.45\linewidth}
      \begin{equation*}
        \begin{tikzcd}[cramped, row sep=small, column sep=small]
          & \bullet \\
          & \bullet \\
          & \bullet \\
          \bullet && \bullet \\
          & \bullet
          \arrow[from=1-2, to=2-2]
          \arrow["\sim", from=2-2, to=3-2]
          \arrow[from=3-2, to=4-1]
          \arrow[from=3-2, to=4-3]
          \arrow[from=4-1, to=5-2]
          \arrow[from=4-3, to=5-2]
        \end{tikzcd}
      \end{equation*}
    \end{minipage}
    \begin{minipage}{0.45\linewidth}
      \begin{equation*}
        \begin{tikzcd}[cramped, row sep=small, column sep=small]
          \bullet &&&& \bullet \\
          \bullet & \bullet & \bullet & \bullet & \bullet \\
          \bullet &&&& \bullet
          \arrow[from=2-2, to=1-1]
          \arrow[from=2-2, to=2-1]
          \arrow[from=2-2, to=3-1]
          \arrow[from=2-3, to=2-2]
          \arrow[from=2-3, to=2-4]
          \arrow["\sim", from=2-4, to=1-5]
          \arrow["\sim"{description}, from=2-4, to=2-5]
          \arrow["\sim"', from=2-4, to=3-5]
        \end{tikzcd}
      \end{equation*}
    \end{minipage}
  \end{center}
  \begin{center}
    \begin{minipage}{0.45\linewidth}
      \begin{equation*}
        \begin{tikzcd}[cramped, row sep=small, column sep=small]
          \bullet & \bullet & \bullet & \bullet & \bullet
          \arrow["\sim"', from=1-2, to=1-1]
          \arrow[from=1-3, to=1-2]
          \arrow[from=1-3, to=1-4]
          \arrow[from=1-4, to=1-5]
        \end{tikzcd}
      \end{equation*}
    \end{minipage}
    \begin{minipage}{0.45\linewidth}
      \begin{equation*}
        \begin{tikzcd}[cramped, row sep=small, column sep=small]
          && \bullet \\
          & \bullet && \bullet \\
          && \bullet \\
          \bullet && \bullet && \bullet
          \arrow[from=2-2, to=1-3]
          \arrow[from=2-2, to=4-1]
          \arrow[from=2-4, to=1-3]
          \arrow[from=2-4, to=4-5]
          \arrow[from=3-3, to=2-2]
          \arrow[from=3-3, to=2-4]
          \arrow[from=3-3, to=4-3]
          \arrow["\sim", from=4-3, to=4-1]
          \arrow["\sim"', from=4-3, to=4-5]
        \end{tikzcd}
      \end{equation*}
    \end{minipage}
  \end{center}
\end{example}
\begin{example}\label{ex:no-htpy-psfw}
  Unfortunately, \Cref{thm:htpy-psfw} does not apply to the following shapes of
  $(\McI,\McW_\McI)$.
  \begin{center}
    \begin{minipage}{0.45\linewidth}
      \begin{equation*}
        \begin{tikzcd}[cramped, row sep=small, column sep=small]
          0 & 01 & 1
          \arrow["\sim"', from=1-2, to=1-1]
          \arrow[from=1-2, to=1-3]
        \end{tikzcd}
      \end{equation*}
    \end{minipage}
    \begin{minipage}{0.45\linewidth}
      \begin{equation*}
        \begin{tikzcd}[cramped, row sep=small, column sep=small]
          0 & 1 & 2
          \arrow[from=1-1, to=1-2]
          \arrow["\sim", curve={height=-12pt}, from=1-1, to=1-3]
          \arrow[from=1-2, to=1-3]
        \end{tikzcd}
      \end{equation*}
    \end{minipage}
  \end{center}

  However, this is also somewhat expected.
  For instance, in the case of
  \begin{tikzcd}[cramped, row sep=small, column sep=small] 0 & 01 \ar[l,
    "\sim"'] \ar[r] & 1 \end{tikzcd}, we can take $\bC$ to be $\Set$ where
  the weak equivalences are the bijections and fibrations are all the maps.
  Then, given homotopical spans $A,B$, the exponential span $[A,B]$ has
  component at 0 as $\Set(A_0,B_0)$ and component at 01 as the set of pairs of
  commutative squares $(s_0,s_1)$ where
  $s_i \in \Set^\to(A_{01} \to A_i, B_{01} \to B_i)$ such that the 01-component
  of both squares agree.
  \begin{equation*}
    \begin{tikzcd}[cramped]
      {A_0} & {A_{01}} & {A_1} \\
      {B_0} & {B_{01}} & {B_1}
      \arrow[dashed, from=1-1, to=2-1]
      \arrow["\cong"', from=1-2, to=1-1]
      \arrow[from=1-2, to=1-3]
      \arrow[dashed, from=1-2, to=2-2]
      \arrow[dashed, from=1-3, to=2-3]
      \arrow["\cong", from=2-2, to=2-1]
      \arrow[from=2-2, to=2-3]
    \end{tikzcd}
  \end{equation*}
  Although every square $s_0 \in \Set^\to(A_{01} \to A_0, B_{01} \to B_0)$ is
  completely determined by its 0-component, it does not determine uniquely a
  square $s_1 \in \Set^\to(A_{01} \to A_1, B_{01} \to B_1)$.
  In particular, although the 01-component of the $s_1$-square is fixed by the
  $s_0$-square, the $s_0$-square has no influence over the 1-component of the
  $s_1$-square.
\end{example}

\printbibliography

\end{document}